\documentclass{amsart}
\usepackage[latin1]{inputenc}
\usepackage{epsfig}
\usepackage{color}
\usepackage[british,english]{babel}
\usepackage{amsthm}
\usepackage{amsmath}
\usepackage{amsfonts}
\usepackage{amssymb}
\usepackage{graphicx}

\newtheorem{corollary}{Corollary}[section]

\newtheorem{lemma}[corollary]{Lemma}
\newtheorem{prp}[corollary]{Proposition}

\newtheorem{thm}[corollary]{Theorem}

\newfont{\sBlackboard}{msbm10 scaled 900}

\newcommand{\mylabel}[1]{\label{#1}
            \ifx\undefined\stillediting
            \else \fbox{$#1$}\fi }
\newcommand{\BE}{\begin{equation}}

\newcommand{\EEQ}{\end{equation}}
\newcommand{\rfb}[1]{\mbox{\rm
   (\ref{#1})}\ifx\undefined\stillediting\else:\fbox{$#1$}\fi}

\newfont{\Blackboard}{msbm10 scaled 1200}

\newfont{\roma}{cmr10 scaled 1200}
\def\RR{{\mathbb R}}

\newcommand{\mm}    {{\hbox{\hskip 0.5pt}}}

\newcommand{\bluff} {{\hbox{\raise 15pt \hbox{\mm}}}}

scaled\magstep2

\makeatletter
\def\section{\@startsection {section}{1}{\z@}{-3.5ex plus -1ex minus
    -.2ex}{2.3ex plus .2ex}{\large\bf}}
\makeatother
\def\be{\begin{equation}}
\def\ee{\end{equation}}

\begin{document}
\title[Nonuniformly elliptic problems]{Low perturbations for a class of \\ nonuniformly elliptic problems}
\author[A. Bahrouni]{Anouar Bahrouni}
\address[A. Bahrouni]{Mathematics Department, University of Monastir,
Faculty of Sciences, 5019 Monastir, Tunisia} \email{\tt
bahrounianouar@yahoo.fr}
\author[D.D. Repov\v{s}]{Du\v{s}an D. Repov\v{s}}
\address[D.D. Repov\v{s}]{Faculty of Education and Faculty of Mathematics and Physics, University of Ljubljana, \& Institute of Mathematics, Physics and Mechanics, 1000 Ljubljana, Slovenia}
\email{\tt dusan.repovs@guest.arnes.si}
\begin{abstract} We introduce and study a new functional which was motivated by our paper
 on the Caffarelli-Kohn-Nirenberg inequality with variable exponent
 (Bahrouni, R\u adulescu \& Repov\v{s}, Nonlinearity 31 (2018), 1518-1534). We also study the eigenvalue problem for equations involving this new functional.
\end{abstract}
\maketitle

\section{Introduction}

The Caffarelli-Kohn-Nirenberg inequality plays an important role in
studying various problems of mathematical physics, spectral theory,
analysis of linear and nonlinear PDEs, harmonic analysis, and
stochastic analysis. We refer to 
Chaudhuri \& Ramaswamy~\cite{ad},
Baroni, Colombo \& Mingione~\cite{mingi1},
Colasuonno \& Pucci~\cite{patrizia1},
and
Colombo \& Mingione~\cite{mingi2}
for
relevant applications of the Caffarelli-Kohn-Nirenberg inequality.

 Let $\Omega\subset\RR^N$ ($N\geq 2$) be a bounded domain with smooth
boundary.  The following Caffarelli-Kohn-Nirenberg inequality
(see Caffarelli, Kohn \& Nirenberg~\cite{caf}) establishes that given $p\in(1,N)$ and real numbers
$a,\,b$, and $q$ such that
$$-\infty<a<\frac{N-p}{p},\;\;\; a\leq b\leq a+1,\;\;\; q=\frac{Np}{N-p(1+a-b)}\,,$$
 there is a positive constant $C_{a,b}$ such that
 for every  $u\in C_c^1(\Omega)$,
\begin{equation}\label{CKNineq}\left(\int_{\Omega}|x|^{-bq}|u|^q\;dx\right)^{p/q}\leq C_{a,b}\int_{\Omega}|x|^{-ap}|\nabla u|^p\;dx\,.\end{equation}

This inequality has been extensively studied (see, e.g.
 Abdellaoui \& Peral~\cite{abdellaui},
Chaudhuri \& Ramaswamy~\cite{ad},
Bahrouni, R\u adulescu \& Repov\v{s}~\cite{annon},
Catrina \& Wang~\cite{wang},
and
Mih\u ailescu,  R\u adulescu \& Stancu~\cite{mh},
 and the references therein). 

In
particular,   Bahrouni, R\u adulescu \& Repov\v{s}~\cite{annon}  gives a new version of
the Caffarelli-Kohn-Nirenberg inequality with variable exponent.
The next theorem is proved under the following assumptions:
let $\Omega\subset\RR^N$ ($N\geq 2$) be
a bounded
domain with smooth boundary and
suppose  that
the following
 hypotheses 
 are satisfied
 \smallskip

(A) $a:\overline{\Omega}\rightarrow\RR$ is a function of class $C^{1}$ and there exist 
$x_{0}\in \Omega$, 
$r>0$, 
and 
 $s \in (1,+\infty)$ 
 such that:
 \begin{enumerate}
\item  $|a(x)|\neq 0,$
 for every 
 $ x\in \overline{\Omega}\setminus \{x_{0}\};$
 \item
 $ |a(x)|\geq |x-x_{0}|^{s},$
 for every 
 $ x\in B(x_{0},r);$
\end{enumerate}

(P) $p:\overline{\Omega}\rightarrow\RR$ is a function of class $C^{1}$ 
and
 $2<p(x)<N$ 
 for every 
$x\in\Omega.$ 
\begin{thm}\label{cafan}{\rm(see {Bahrouni, R\u adulescu \& Repov\v{s}~\cite{annon})}}
 Suppose
 that hypotheses $(A)$
and $(P)$ are satisfied. Let  $\Omega\subset\RR^N$ ($N\geq 2$) be a bounded
domain with smooth boundary. Then there exists a positive constant
$\beta$ such that
\begin{align*}
\displaystyle \int_{\Omega}|a(x)|^{p(x)}|u(x)|^{p(x)}dx &\leq \beta
 \int_{\Omega}|a(x)|^{p(x)-1}||\nabla a(x)||u(x)|^{p(x)}dx\\
&+\beta\left(\displaystyle \int_{\Omega}|a(x)|^{p(x)}|\nabla
u(x)|^{p(x)}dx+\displaystyle \int_{\Omega}|a(x)|^{p(x)}|\nabla p(x)|
|u(x)|^{p(x)+1}dx\right)\\
&+\beta \displaystyle \int_{\Omega}|a(x)|^{p(x)-1}|\nabla p(x)|
|u(x)|^{p(x)-1}dx.
\end{align*}
 for every  $u\in C_{c}^{1}(\Omega)$.
\end{thm}
Motivated by  Bahrouni, R\u adulescu \& Repov\v{s}~\cite{annon}, we introduce
and study in the present paper 
a
new functional  $T:E_{1}\rightarrow \mathbb{R}$
 via the
Caffarelli-Kohn-Nirenberg inequality, in the framework  of variable
exponents. More precisely, we study the eigenvalue problem in which
functional $T$
 is present. Our main result is Theorem~\ref{mainval} and 
 we prove
 it
  in Section~5.

\section{Function spaces with variable exponent}

We recall some necessary properties of variable
exponent spaces. We refer to 
Hajek, Santalucia, Vanderwerff \& Zizler~\cite{china},
Musielak~\cite{musi},
Papageorgiou,  R\u adulescu \& Repov\v{s}~\cite{PRR},
R\u adulescu~\cite{radnla},
R\u adulescu~\cite{radom},
and
R\u adulescu \& Repov\v{s}~\cite{Rad},
 and the
references therein.

 Consider the set
$$C_+(\overline\Omega)=\{p\in C(\overline\Omega)\mid \;p(x)>1\;{\rm
for}\; {\rm all}\;x\in\overline\Omega\}.$$ 

For any $p\in
C_+(\overline\Omega)$, let
    $$p^+=\sup_{x\in\overline{\Omega}}p(x)\qquad\mbox{and}\qquad p^-=
    \inf_{x\in\overline{\Omega}}p(x),$$ and define the {\it variable exponent Lebesgue
    space} as follows
    $$L^{p(x)}(\Omega)=\left\{u\mid\ u\ \mbox{is
measurable real-valued function such that }
 \int_\Omega|u(x)|^{p(x)}\;dx<\infty\right\},$$ with the {\it Luxemburg norm}
 $$|u|_{p(x)}=\inf\left\{\mu>0\mid \;\int_\Omega\left|
 \frac{u(x)}{\mu}\right|^{p(x)}\;dx\leq 1\right\}.$$
 
  We recall that the variable exponent Lebesgue spaces are separable and reflexive Banach spaces if and only if $1 < p^-\leq
 p^+<\infty$,
  and continuous functions with compact support
 are dense in $L^{p(x)}(\Omega)$ if $p^{+}<\infty$.
 
 Let $L^{q(x)}(\Omega)$ denote the conjugate space of
 $L^{p(x)}(\Omega)$, where $1/p(x)+1/q(x)=1$. If $u\in
 L^{p(x)}(\Omega)$ and $v\in L^{q(x)}(\Omega)$ then  the following
 H\"older-type inequality holds:
 \begin{equation}\label{Hol}
 \left|\int_\Omega uv\;dx\right|\leq\left(\frac{1}{p^-}+
 \frac{1}{q^-}\right)|u|_{p(x)}|v|_{q(x)}\,.
 \end{equation}
 
 An important role in manipulating the generalized Lebesgue-Sobolev spaces is
played by the $p(.)-$modular of the $L^{p(x)}(\Omega)$ space, which
is the mapping $\rho: L^{p(x)}(\Omega) \rightarrow \mathbb{R}$
defined by
$$\rho(u)=\displaystyle \int_{\Omega}|u|^{p(x)}dx.$$
 \begin{prp}\label{pr1}{\rm{(see R\u adulescu \& Repov\v{s}~\cite{Rad})}}
 The following properties hold\\
 $(i)$ $|u|_{p(x)}<1 (\mbox{resp}., =1;>1)\Leftrightarrow \rho(u)<1(\mbox{resp}., =1;>1)$;\\
 $(ii)$ $|u|_{p(x)}>1 \Rightarrow |u|_{p(x)}^{p^{-}}\leq \rho(u) \leq
 |u|_{p(x)}^{p^{+}}$; and\\
 $(iii)$ $|u|_{p(x)}<1 \Rightarrow |u|_{p(x)}^{p^{+}}\leq \rho(u)
 \leq |u|_{p(x)}^{p^{-}}$.
 \end{prp}
 \begin{prp}\label{pr2}{\rm{(see R\u adulescu \& Repov\v{s}~\cite{Rad})}}
 If $u,u_{n}\in L^{p(x)}(\Omega)$ and  $n\in \mathbb{N}$, then the
   following statements are equivalent\\
 $(1)$ $\displaystyle \lim_{n\rightarrow +\infty}
 |u_{n}-u|_{p(x)}=0$.\\
 $(2)$  $\displaystyle \lim_{n\rightarrow +\infty} \rho(
 u_{n}-u)=0.$\\
 $(3)$ $u_{n}\rightarrow u$ in measure in $\Omega$ and
 $\displaystyle \lim_{n\rightarrow +\infty}\rho(u_{n})=\rho(u).$
 \end{prp}
 We define the {\it variable exponent  Sobolev space} by
  $$
  W^{1,p(x)}(\Omega)=\{u\in L^{p(x)}(\Omega)\mid \; |\nabla u|\in
  L^{p(x)} (\Omega) \}.
  $$
  
On $W^{1,p(x)}(\Omega)$ we  consider the following norm
  $$
  \|u\|_{p(x)}=|u|_{p(x)}+|\nabla u|_{p(x)}.
  $$
  
  Then $W^{1,p(x)}(\Omega)$ is a reflexive separable Banach space.
  
\section{Functional $T$}

We shall introduce a new functional $T:E_{1}\rightarrow \mathbb{R}$ motivated by the
Caffarelli-Kohn-Nirenberg inequality obtained in Bahrouni, R\u adulescu \& Repov\v{s}~\cite{annon}. 

We
denote
 by $E_{1}$ the closure of
$C_{c}^{1}(\Omega)$ under the norm
$$\begin{array}{ll}
\|u\|=&\displaystyle ||B(x)|^{\frac{1}{p(x)}}\nabla u(x)|_{p(x)}+|A(x)^{\frac{1}{p(x)}}u(x)|_{p(x)}+\\
&\displaystyle ||D(x)|^{\frac{1}{p(x)+1}}u(x)|_{p(x)+1}+   ||C(x)|^{\frac{1}{p(x)-1}}u(x)|_{p(x)-1},\end{array}$$
where the potentials $A$, $B$, $C$, and $D$ are defined by
\begin{equation}\label{ABCD}
 \left\{\begin{array}{ll}
 &\displaystyle A(x)=|a(x)|^{p(x)-1}|\nabla a(x)|\\
 &\displaystyle B(x)=|a(x)|^{p(x)}\\
 &\displaystyle C(x)=|a(x)|^{p(x)-1}|\nabla p(x)|\\
 &\displaystyle D(x)=B(x)|\nabla p(x)|.
 \end{array}\right.
 \end{equation}
 
We  
now 
define $T:E_{1}\rightarrow \mathbb{R}$ as follows
\begin{align*}
T(u)&=\displaystyle
\int_{\Omega} \frac{B(x)}{p(x)} |\nabla u(x)|^{p(x)}dx +
\displaystyle \int_{\Omega}\frac{A(x)}{p(x)}|u(x)|^{p(x)}dx \\&+
\displaystyle \int_{\Omega}\frac{D(x)}{p(x)+1}|u(x)|^{p(x)+1}dx+
\displaystyle \int_{\Omega}\frac{C(x)}{p(x)-1}|u(x)|^{p(x)-1}dx.
\end{align*}

The following properties of 
$T$ will be useful in the sequel.
\begin{lemma}\label{der}
Suppose that hypotheses $(A)$ and $(P)$ are satisfied. Then the functional $T$
is well-defined on $E_{1}$. Moreover, $T\in C^{1}(E_{1},\mathbb{R})$
with the derivative given by
\begin{align*}
\langle L(u),v\rangle=\langle T'(u),v\rangle&= \displaystyle
\int_{\Omega} B(x) |\nabla u(x)|^{p(x)-2}\nabla u(x) \nabla v(x)dx +
\displaystyle \int_{\Omega}A(x)|u(x)|^{p(x)-2}u(x) v(x)dx \\&+
\displaystyle \int_{\Omega}D(x)|u(x)|^{p(x)-1}u(x) v(x)dx+
\displaystyle \int_{\Omega}C(x)|u(x)|^{p(x)-3}u(x)v(x)dx,
\end{align*}
for every $u,v\in E_{1}$.
\end{lemma}

\begin{proof}
The proof is standard, see R\u adulescu \& Repov\v{s}~\cite{Rad}.
\end{proof}
\begin{lemma}\label{sss}
Suppose that hypotheses $(A)$ and $(P)$ are satisfied. Then the following properties hold\\
$(i)$ $L:E_{1}\rightarrow E^{\ast}_{1}$ is a continuous, bounded and
strictly monotone operator;\\
$(ii)$ $L$ is a mapping of type $(S_{+})$, i.e. if
$u_{n}\rightharpoonup u$ in $E_{1}$ and 
$$\displaystyle \limsup_{n
\rightarrow +\infty}
\langle L(u_{n})-L(u), u_{n}-u\rangle \leq
0,$$
 then $u_{n}\rightarrow u$ in $E_{1}$.
\end{lemma}
\begin{proof}
(i) Evidently, $L$ is a bounded operator. Recall the following Simon inequalities
(see Simon~\cite{simon}):
\begin{equation}\label{s}
\begin{cases}
\left|x-y\right|^{p}\leq
c_{p}\left(\left|x\right|^{p-2}x-\left|y\right|^{p-2}y\right).(x-y)
\; \; \; \; \;  \; \; \; \; \;   \; \; \; \; \; \; \; \; \; \; \; \; \; \; \; \; \; \mbox{for} \ \ p\geq 2 \\
 \left|x-y\right|^{p} \leq C_{p}\left[\left(\left|x\right|^{p-2}x-\left|y\right|^{p-2}y\right).(x-y)\right]^{\frac{p}{2}}
 \left(\left|x\right|^{p}+\left|y\right|^{p}\right)^{\frac{2-p}{2}} \; \; \mbox{for} \ \ 1<p< 2,
\end{cases}
\end{equation}
for every  $x,y\in \mathbb{R}^{N}$, where 
$$c_{p}=(\frac{1}{2})^{-p} \ \ 
\mbox{and} \ \ 
C_{p}=\frac{1}{p-1}.$$

 Using inequalities
 \eqref{s} and recalling that $2<p^{-}$, we can prove that $L$ is 
a strictly monotone operator.

(ii) The proof is identical to the proof of Theorem $3.1$ in
Fan \& Zhang~\cite{Fan}.
\end{proof}

\section{Main theorem} 

We recall our Compactness Lemma:

\begin{lemma}\label{com} {\rm{(see Bahrouni, R\u adulescu \& Repov\v{s}~\cite{annon})}}
Suppose that hypotheses $(A)$ and $(P)$ are satisfied
and
that $p^{-}>1+s$. Then $E_{1}$ is compactly embeddable
 in
$L^{q}(\Omega)$ for each $q\in (1,\frac{Np^{-}}{N+sp^{+}})$.
Moreover, the same conclusion holds if we  replace $L^{q}(\Omega)$
by $L^{q(x)}(\Omega)$, provided that $q^+< \frac{Np^{-}}{N+sp^{+}}$.
\end{lemma}

We are concerned with the following nonhomogeneous
problem
\begin{equation}\label{val}
  \left\{\begin{array}{ll} &\displaystyle -{\rm div}\,(B(x)|\nabla
u|^{p(x)-2}\nabla
u)+(A(x)|u|^{p(x)-2}+C(x)|u|^{p(x)-3})u=\\
&\displaystyle (\lambda |u|^{q(x)-2}-D(x)|u|^{p(x)-1})u
\quad\mbox{in
}\phantom{\partial}\Omega,\\
&u=0 \quad\mbox{on}\ \partial\Omega\,,
\end{array}\right.
\end{equation}
where $\lambda >0$ is a real number and $q$ is continuous on
$\overline{\Omega}$. We assume that $q$ satisfies the following
basic inequalities
$$ (Q) \ \ \  \ \ \ \  \ \ \ \ \  \ \ \ \ \ \ \ \ \ \ \    1<\min_{x\in \overline{\Omega}}q(x)<\min_{x\in \overline{\Omega}}(p(x)-1)<\max_{x\in \overline{\Omega}}q(x)<\frac{Np^{-}}{N+sp^{+}}. $$

We can now state the main result of this paper.

\begin{thm}\label{mainval}
 Suppose that all hypotheses of Lemma \ref{com} are satisfied and
  that inequalities $(Q)$ hold. Then there exists
 $\lambda_{0}>0$ such that every $\lambda \in (0,\lambda_{0})$ is
 an eigenvalue for problem \eqref{val}.
\end{thm}

In order to prove Theorem \ref{mainval} (this will be done in the Section~5), we shall need some preliminary results. We begin by defining the functional $I_{\lambda}:
E_{1}\rightarrow \mathbb{R},$ 
\begin{align*}
I_{\lambda}(u)&=\displaystyle \int_{\Omega} \frac{B(x)}{p(x)}
|\nabla u(x)|^{p(x)}dx+ \displaystyle \int_{\Omega}
\frac{A(x)}{p(x)}|u(x)|^{p(x)}dx+\displaystyle \int_{\Omega}
\frac{C(x)}{p(x)-1}|u(x)|^{p(x)-1}dx\\
&+ \displaystyle \int_{\Omega} \frac{D(x)}{p(x)+1}|u(x)|^{p(x)+1}dx
-\lambda \displaystyle \int_{\Omega}\frac{ |u(x)|^{q(x)}}{q(x)}dx.
\end{align*}

Standard argument shows that $I_{\lambda}\in C^{1}(E_{1},
\mathbb{R})$ and
\begin{align*}
\langle I_{\lambda}'(u),v\rangle&= \displaystyle \int_{\Omega} B(x)
|\nabla u(x)|^{p(x)-2}\nabla u(x) \nabla v(x)dx + \displaystyle
\int_{\Omega}A(x)|u(x)|^{p(x)-2}u(x) v(x)dx \\&+ \displaystyle
\int_{\Omega}D(x)|u(x)|^{p(x)-1}u(x) v(x)dx+ \displaystyle
\int_{\Omega}C(x)|u(x)|^{p(x)-3}u(x)v(x)dx\\& - \lambda \int_{\Omega}
|u(x)|^{q(x)-2}u(x) v(x),
\end{align*}

for every  $u,v \in E_{1}$. 

Thus the weak solutions of problem
\eqref{val} coincide with the critical points of $I_{\lambda}$.
\begin{lemma}\label{g1}
Suppose that all hypotheses of Theorem \ref{mainval} are satisfied.
Then there exists $\lambda_{0}>0$ such that for any $\lambda \in
(0,\lambda_{0})$ there exist $\rho,\alpha>0$ such that
$$I_{\lambda}(u)\geq \alpha \ \ \mbox{for any }\ \ u\in E_{1} \ \
\mbox{with} \ \ \|u\|=\rho.$$
\end{lemma}
\begin{proof}
By Lemma \ref{com}, there exists $\beta >0$ such that
$$\left|u\right|_{r(x)}\leq \beta\left\|u\right\|, \ \ \mbox{for every} \ \ u\in E_{1} \ \ \mbox{and} \ \ r^{+}\in (1,
\frac{Np^{-}}{N+sp^{+}}).$$ 

We fix $\rho \in (0,
\min(1,\frac{1}{\beta}))$. Invoking Proposition \ref{pr1}, for every
$u\in E_{1} \ \ \mbox{with} \ \ \|u\|=\rho$, we can get
$$|u|_{q(x)}<1.$$

Combining the above relations and Proposition \ref{pr1}, for any
$u\in E_{1}$ with $\|u\|=\rho$, we can then deduce that
\begin{equation}\label{n}
\begin{array}{ll}
I_{\lambda}(u)&\displaystyle\geq  \frac{1}{p^{+}}\left(\displaystyle
\int_{\Omega} B(x) |\nabla u(x)|^{p(x)}dx+\displaystyle
\int_{\Omega}
A(x)|u(x)|^{p(x)}dx\right)\\
&\displaystyle + \frac{1}{p^{+}+1}\displaystyle \int_{\Omega}
D(x) | u(x)|^{p(x)+1}dx\\
&\displaystyle +\frac{1}{p^{+}-1} \displaystyle
\int_{\Omega}C(x)|u(x)|^{p(x)-1}dx-
\frac{\lambda}{q^{-}}\displaystyle \int_{\Omega}
|u(x)|^{q(x)}dx
 \\
&\displaystyle\geq \frac{1}{4^{p^{+}}(p^{+}+1)}
\|u\|^{p^{+}+1}-\lambda \frac{\beta^{q^{-}}}{q^{-}}\|u\|^{q^{-}}\\&
\geq \displaystyle \frac{1}{4^{p^{+}}(p^{+}+1)}
\rho^{p^{+}+1}-\lambda \frac{\beta^{q^{-}}}{q^{-}}\rho^{q^{-}}\\
&=\rho^{q^{-}}(
\frac{1}{4^{p^{+}}(p^{+}+1)}\rho^{p^{+}+1-q^{-}}-\lambda
\frac{\beta^{q^{-}}}{q^{-}}).
\end{array}
\end{equation}

Put
$\lambda_{0}=\frac{\rho^{p^{+}+1-q^{-}}}{4^{p^{+}}(2p^{+}+2)}\frac{{q^{-}}}{\beta^{q^{-}}}$.
It now follows from \eqref{n} that  for any $\lambda
\in(0,\lambda_{0})$,
$$I_{\lambda} (u)\geq \alpha \ \ \mbox{with} \ \ \|u\|=\rho,$$

and $\alpha=\frac{\rho^{p^{+}+1}}{4^{p^{+}}(2p^{+}+2)}>0.$ This
completes the proof of Lemma~\ref{g1}.
\end{proof}

\begin{lemma}\label{g2}
Suppose that all hypotheses of Theorem \ref{mainval} are satisfied. Then
there exists $\varphi\in E_{1}$ such that $\varphi> 0$ and
$I_{\lambda}(t\varphi)<0$, for small enough $t$.
\end{lemma}
\begin{proof}
By virtue of hypotheses $(P)$ and $(Q)$, there exist
$\epsilon_{0}>0$ and $\Omega_{0}\subset \Omega$ such that
\begin{equation}\label{e1}
q(x)<q^{-}+\epsilon_{0}<p^{-}-1, \ \ \mbox{for every } \ \ x\in
\Omega_{0}.
\end{equation}

 Let $\varphi \in C^{\infty}_{0}(\Omega)$
such that $\overline{\Omega_{0}}\subset \mbox{supp}(\varphi)$, $\varphi=1$
for every  $x \in \overline{\Omega_{0}}$ and $ 0 \leq \varphi \leq 1$
in $\Omega$. It then follows that for $t\in (0,1)$,
\begin{align}\label{e2}
I_{\lambda}(t\varphi)&=\displaystyle \int_{\Omega}
\frac{t^{p(x)}B(x)}{p(x)} |\nabla \varphi(x)|^{p(x)}dx+
\displaystyle \int_{\Omega}
\frac{t^{p(x)}A(x)}{p(x)}|\varphi(x)|^{p(x)}dx+\displaystyle
\int_{\Omega}\frac{t^{p(x)-1}C(x)}{p(x)-1}|\varphi|^{p(x)-1}dx
\nonumber\\&+\displaystyle \int_{\Omega}
\frac{t^{p(x)+1}D(x)}{p(x)+1}|\varphi(x)|^{p(x)+1}dx
-\lambda \displaystyle \int_{\Omega}t^{q(x)}\frac{ |\varphi(x)|^{q(x)}}{q(x)}dx \nonumber\\
&\leq \frac{t^{p^{-}-1}}{p^{-}-1}(\displaystyle \int_{\Omega}
\frac{B(x)}{p(x)} |\nabla \varphi(x)|^{p(x)}dx+ \displaystyle
\int_{\Omega} \frac{A(x)}{p(x)}|\varphi(x)|^{p(x)}dx+\displaystyle
\int_{\Omega}\frac{C(x)}{p(x)-1}|\varphi|^{p(x)-1}dx \nonumber
\\&+\displaystyle \int_{\Omega}
\frac{D(x)}{p(x)+1}|\varphi(x)|^{p(x)+1}dx)
-\lambda t^{q^{-}+\epsilon_{0}} \displaystyle \int_{\Omega}\frac{
|\varphi(x)|^{q(x)}}{q(x)}dx.
\end{align}

Combining \eqref{e1} and \eqref{e2}, we finally arrive at the
desired conclusion. 

This
completes the proof of Lemma~\ref{g2}.
\end{proof}

 \section{Proof of Theorem~\ref{mainval}}
 
In the last section we shall prove the main theorem of this paper.

Let $\lambda_{0}$ be defined as in Lemma \ref{g1} and choose any
$\lambda \in (0,\lambda_{0})$. 

Again, invoking Lemma \ref{g1}, we
can deduce that
\begin{equation}\label{eq3}
\displaystyle \inf_{u\in \partial B(0,\rho)}I_{\lambda}(u)>0.
\end{equation}

On the other hand, by Lemma \ref{g2}, there exists $\varphi\in
E_{1}$ such that 
$$I_{\lambda}(t\varphi)<0 \ \ \rm{ for \ every \  small \ enough} \ 
t>0.$$ 

Moreover, by Proposition \ref{pr1}, when $\|u\|<\rho$, we
have
$$I_{\lambda}(u)\geq \frac{1}{4^{p^{+}}(p^{+}+1)}\|u\|^{p^{+}+1}-c \|u\|^{q^{-}},$$

where $c$ is a positive constant. It follows that
$$-\infty<m=\displaystyle \inf_{u\in B(0,\rho)}I_{\lambda}(u)<0.$$

Applying Ekeland's variational principle to the functional
$$I_{\lambda}: B(0,\rho)\rightarrow \mathbb{R},$$

 we can find a (PS)
sequence $(u_{n})\in B(0,\rho)$, that is,
$$I_{\lambda}(u_{n})\rightarrow m \ \ \mbox{and}\ \ I_{\lambda}^{'}(u_{n})\rightarrow 0.$$

It is clear that $(u_{n})$ is bounded in $E_{1}$. Thus there exists
$u\in E_{1}$ such that, up to a subsequence,
$$(u_{n})
\rightharpoonup u \ \  \rm{in} \ \  E_{1}.$$ 

Using Theorem \ref{com}, we see that
$(u_{n})$ strongly converges to $u$ in $L^{q(x)}(\Omega)$. 

So, by
the H\"older inequality and Proposition \ref{pr2}, we can obtain the
following
$$\displaystyle \lim_{n\rightarrow  +\infty}\displaystyle \int_{\Omega}|u_{n}|^{q(x)-2}u_{n}(u_{n}-u)dx=
\displaystyle \lim_{n\rightarrow  +\infty}\displaystyle
\int_{\Omega}|u|^{q(x)-2}u(u_{n}-u)dx=0.$$

 On the other hand, since
$(u_{n})$ is a (PS) sequence, we can also infer that
$$\lim_{n\rightarrow  +\infty}\langle I^{'}_{\lambda}(u_{n})- I^{'}_{\lambda}(u), u_{n}-u\rangle=0.$$

Combining the above pieces of information with Lemma \ref{sss}, we
can now conclude that 
$$u_{n}\rightarrow u \ \  \rm{in} \ \  E_{1}.$$ 

Therefore 
$$I_{\lambda}(u)=m<0 \ \ \mbox{and} \ \ I_{\lambda}^{'}(u)=0.$$

We have thus shown that $u$ is a nontrivial weak solution for
problem \eqref{val} and that every $\lambda \in (0,\lambda_{0})$ is
an eigenvalue of problem \eqref{val}. 

This completes the proof of Theorem~\ref{mainval}.  \qed\\

\section*{Acknowledgements} 
The second author was supported by the Slovenian Research Agency grants P1-0292, J1-7025, J1-8131,
N1-0064,  N1-0083, and N1-0114. We thank the referee for comments and suggestions.

\vfill\eject

\end{document}